\documentclass[11pt]{article}

\usepackage[margin=1in]{geometry}

\usepackage[utf8]{inputenc}
\usepackage[T1]{fontenc}
\usepackage[english]{babel}

\usepackage{amsmath,amssymb,amsthm,mathtools}
\usepackage{mathrsfs}
\usepackage{enumitem}
\usepackage{graphicx}
\usepackage[all]{xy}
\usepackage{booktabs}
\usepackage{tabularx}
\usepackage{xcolor}
\usepackage{url}
\usepackage{hyperref}
\usepackage{authblk}

\theoremstyle{plain}
\newtheorem{theorem}{Theorem}[section]
\newtheorem{proposition}[theorem]{Proposition}
\newtheorem{corollary}[theorem]{Corollary}

\theoremstyle{definition}
\newtheorem{definition}[theorem]{Definition}
\newtheorem{example}[theorem]{Example}

\theoremstyle{remark}
\newtheorem{remark}[theorem]{Remark}

\newtheorem{notation}[theorem]{Notation}

\newcommand{\id}{\operatorname{id}}

\DeclareMathOperator{\Glike}{Glike}
\DeclareMathOperator{\Slike}{Slike}
\DeclareMathOperator{\Conj}{Conj}
\DeclareMathOperator{\Halo}{Halo}

\DeclareMathOperator{\Prim}{Prim}
\DeclareMathOperator{\Adj}{Adj}

\newcommand{\fnm}[1]{#1}
\newcommand{\sur}[1]{#1}
\newcommand{\jyear}[1]{}
\newcommand{\equalcont}[1]{}
\newcommand{\orgdiv}[1]{#1}
\newcommand{\orgname}[1]{#1}
\newcommand{\orgaddress}[1]{#1}
\newcommand{\street}[1]{#1}
\newcommand{\city}[1]{#1}
\newcommand{\country}[1]{#1}

\title{Cocommutative Hopf Dialgebras and Rack Combinatorics}

\author[1]{\fnm{Jos\'e Gregorio }\sur{Rodr\'iguez-Nieto}\thanks{\texttt{jgrodrig@unal.edu.co}}}

\author[1]{\fnm{Olga Patricia }\sur{Salazar-D\'iaz}\thanks{\texttt{opsalazard@unal.edu.co}}}

\author[2]{\fnm{Andr\'es }\sur{Sarrazola-Alzate}\thanks{\texttt{andres.sarrazola@eia.edu.co}}}

\author[3]{\fnm{Ra\'ul }\sur{Vel\'asquez}\thanks{\texttt{raul.velasquez@udea.edu.co}}}

\affil[1]{\orgdiv{Departamento de Matem\'aticas}, \orgname{Universidad Nacional de Colombia},
	\orgaddress{\street{Carrera 65 No. 59A-110}, \city{Medell\'in}, \country{Colombia}}}

\affil[2]{\orgname{STEAM School, Universidad EIA},
	\orgaddress{\street{Calle 23 AA Sur Nro. 5-200, Kil\'ometro 2+200 Variante al Aeropuerto Jos\'e Mar\'ia C\'ordova},
		\city{Envigado}, \country{Colombia}}}

\affil[3]{\orgdiv{Instituto de Matem\'aticas}, \orgname{Universidad de Antioquia},
	\orgaddress{\street{Calle 67 No. 53-108}, \city{Medell\'in}, \country{Colombia}}}

\date{}

\begin{document}
	
	
	\maketitle
	
	\begin{abstract}
		We study cocommutative Hopf dialgebras through generalized digroups and rack
		combinatorics. We prove that the rack functor obtained from the adjoint rack
		bialgebra factorizes through the digroup of group-like elements. More precisely,
		for every cocommutative Hopf dialgebra \(A\), the rack of set-like elements of
		its adjoint rack bialgebra is naturally isomorphic to the conjugation rack of
		the digroup \(\Glike(A)\). For finite generalized digroups
		\(D\simeq G\times E\), with \(G\) acting on the halo \(E\), we derive explicit
		formulas for the conjugation rack, its inner group, left-translation cycle
		index, fixed-point polynomial, orbit count and subrack structure. Finally, we
		construct the digroup algebra \(K[D]\), prove that it is a cocommutative Hopf
		dialgebra, and show that \(\Glike(K[D])=D\).
	\end{abstract}
	
	
	
	\section{Introduction}
	
	Hopf algebras provide one of the classical algebraic mechanisms by which
	multiplicative and comultiplicative structures interact. In the cocommutative
	case, the group-like elements of a Hopf algebra form a group, while primitive
	elements form a Lie algebra. This familiar picture suggests a broader question:
	what kind of global and combinatorial structures arise when the associative
	multiplication is replaced by the two products of a dialgebra? The purpose of
	this paper is to study this question for cocommutative Hopf dialgebras and to
	show that their rack-theoretic shadow is controlled by generalized digroups and
	their conjugation racks.
	
	The dialgebraic framework originates in the relationship between associative
	dialgebras and Leibniz algebras. Loday introduced dialgebras as the associative
	counterpart of Leibniz algebras, in the same way that associative algebras give
	rise to Lie algebras through the commutator construction \cite{Loday2001}.
	In this setting, the two products \(\vdash\) and \(\dashv\) encode the Leibniz
	bracket through
	\[
	[x,y]=x\vdash y-y\dashv x.
	\]
	Hopf dialgebras refine this picture by adding coalgebraic data and antipode
	identities adapted to the two products. In the cocommutative case,
	Alexandre--Bordemann--Rivi\`ere--Wagemann showed that every Hopf dialgebra
	carries an adjoint rack bialgebra structure \cite{ABRW1}. This construction is
	the Hopf-dialgebraic starting point of the present work.
	
	Racks and quandles appear naturally in topology, knot theory, Hopf-algebraic
	classification problems, and self-distributive algebra. The knot quandle of
	Joyce and the quandle cohomology invariants of Carter--Jelsovsky--Kamada--
	Langford--Saito are fundamental examples showing how self-distributivity
	encodes geometric information \cite{Joyce1982,Carter2003}. The homotopy and
	cohomological theory of racks was further developed through rack spaces and
	rack cohomology \cite{FennRourkeSanderson1995,EtingofGrana2003}. Racks also
	enter the theory of pointed Hopf algebras through Nichols algebras associated
	with rack data and cocycles \cite{AndruskiewitschGrana2003}. These connections
	motivate the search for explicit algebraic sources of racks and for
	combinatorial invariants attached to them.
	
	Generalized digroups provide one such source. They are algebraic structures
	with two associative products, a possibly non-singleton set of bar-units called
	the halo, and one-sided inverse data relative to bar-units. They were introduced
	and developed as a natural enlargement of digroups and as global objects
	adapted to dialgebraic and Leibniz-type phenomena \cite{SVW,RSV}. A structure
	theorem identifies every generalized digroup, after choosing a bar-unit, with a
	product
	\[
	G\times E,
	\]
	where \(G\) is an ordinary group and \(E\) is the halo equipped with a
	\(G\)-action \cite{RSV4}. This decomposition makes the conjugation rack of a
	generalized digroup completely explicit:
	\[
	(g,\alpha)\triangleright(h,\beta)
	=
	(ghg^{-1},g\bullet\beta).
	\]
	The recent correspondence between racks and generalized digroups gives the
	categorical background for this construction \cite{RSGD}.
	
	The first main contribution of this paper is to show that the rack functor
	attached to a cocommutative Hopf dialgebra factorizes through the digroup of
	group-like elements. More precisely, if \(A\) is a cocommutative Hopf
	dialgebra, then its group-like elements form a digroup under the restricted
	dialgebra products, with inverse induced by the antipode. The adjoint rack
	bialgebra construction, followed by passage to set-like elements, agrees with
	the conjugation rack of this digroup. In functorial form, we obtain a natural
	factorization
	\[
	\mathbf{HopfDialg}
	\xrightarrow{\ \Glike\ }
	\mathbf{Dig}
	\hookrightarrow
	\mathbf{gDig}
	\xrightarrow{\ \Conj\ }
	\mathbf{Rack}.
	\]
	Thus the rack associated with a cocommutative Hopf dialgebra is not merely an
	external self-distributive structure; it is governed by an intermediate
	digroup.
	
	The second main contribution is combinatorial. For a finite generalized
	digroup \(D\simeq G\times E\), the conjugation rack admits an explicit
	description in terms of group conjugation on \(G\) and the action of \(G\) on
	the halo \(E\). We use this description to compute the inner permutation group,
	the left-translation cycle index, a fixed-point polynomial, an orbit formula,
	and a structural criterion for subracks. The cycle-index viewpoint is inspired
	by classical enumerative methods \cite{Read1968,HararyPalmer}, but it is
	adapted here to the family of left translations of a rack rather than to an
	arbitrary permutation group. In particular, if \(c_g\) denotes conjugation by
	\(g\) on \(G\) and \(\rho(g)\) denotes the permutation of \(E\), then the left
	translation by \((g,\alpha)\) is
	\[
	L_{(g,\alpha)}=c_g\times \rho(g),
	\]
	and its cycle structure is determined by the cycle structures of these two
	permutations.
	
	The final part of the paper studies the digroup algebra construction. Starting
	from a digroup \(D\), we form the free vector space \(K[D]\), extend both
	products linearly, and define
	\[
	\Delta(x)=x\otimes x,\qquad
	\varepsilon(x)=1,\qquad
	S(x)=x^{-1}
	\]
	on basis elements. This produces a cocommutative Hopf dialgebra. Moreover, its
	group-like elements recover the original digroup:
	\[
	\Glike(K[D])=D.
	\]
	Consequently, the general factorization theorem specializes to the expected
	identity: the rack associated with \(K[D]\) is precisely the conjugation rack of
	\(D\). This shows that the digroup algebra construction is the dialgebraic
	analogue of the usual group algebra construction, at least from the viewpoint
	of group-like elements and conjugation racks.
	
	The paper is organized as follows. Section~2 fixes the preliminary language on
	coalgebras, dialgebras, racks, generalized digroups and categorical
	conventions. Section~3 recalls cocommutative Hopf dialgebras, rack bialgebras,
	the structure theorem for generalized digroups and the conjugation rack
	construction. Section~4 develops the finite combinatorics of conjugation racks
	of generalized digroups, including inner groups, left-translation cycle
	indices, fixed-point polynomials, orbit formulas and subrack structure.
	Section~5 proves the group-like factorization theorem for cocommutative Hopf
	dialgebras and constructs the digroup algebra \(K[D]\), showing that this
	construction recovers the original digroup and its conjugation rack.
	
	\section{Preliminaries}\label{sec:preliminaries}
	
	Throughout the paper, \(K\) denotes a field of zero characteristic and all tensor products are taken
	over \(K\), unless otherwise stated. We collect here the basic algebraic,
	coalgebraic and categorical conventions used in the sequel.
	
	\subsection{Coalgebraic notation}
	
	Let \((C,\Delta,\varepsilon)\) be a coalgebra over \(K\). We use Sweedler
	notation
	\[
	\Delta(c)=\sum_{(c)}c_{(1)}\otimes c_{(2)}
	\qquad
	(c\in C).
	\]
	An element \(g\in C\) is called \textbf{group-like} if
	\[
	\Delta(g)=g\otimes g,
	\qquad
	\varepsilon(g)=1.
	\]
	The set of group-like elements of \(C\) will be denoted by
	\[
	\Glike(C).
	\]
	If \(\xi\in C\) is a fixed group-like element, an element \(x\in C\) is called
	\textbf{\(\xi\)-primitive} if
	\[
	\Delta(x)=x\otimes \xi+\xi\otimes x.
	\]
	The corresponding space of primitive elements is denoted by
	\[
	\Prim(C)
	:=
	\{x\in C\mid \Delta(x)=x\otimes \xi+\xi\otimes x\}.
	\]
	These conventions are standard in Hopf algebra theory; see
	\cite[Chapter~III]{Kassel1995} and \cite[Chapter~1]{Majid1995}.
	
	A linear map
	\[
	f:C\longrightarrow C'
	\]
	between coalgebras is a \textbf{coalgebra morphism} if
	\[
	\Delta_{C'}\circ f=(f\otimes f)\circ\Delta_C
	\qquad
	\text{and}
	\qquad
	\varepsilon_{C'}\circ f=\varepsilon_C.
	\]
	In particular, coalgebra morphisms preserve group-like elements.
	
	\subsection{Dialgebras and Leibniz brackets}
	
	An \textbf{associative dialgebra}, or simply a \textbf{dialgebra}, is a
	\(K\)-vector space \(A\) endowed with two associative bilinear products
	\[
	\vdash,\dashv:A\otimes A\longrightarrow A
	\]
	satisfying the mixed identities
	\[
	(a\vdash b)\vdash c=(a\dashv b)\vdash c,
	\]
	\[
	a\dashv(b\dashv c)=a\dashv(b\vdash c),
	\]
	and
	\[
	(a\vdash b)\dashv c=a\vdash(b\dashv c)
	\]
	for all \(a,b,c\in A\). Dialgebras were introduced by Loday as the associative
	counterpart of Leibniz algebras; see \cite[Section~1]{Loday2001} and
	\cite[Section~1]{LodayPirashvili1993}. For the operadic viewpoint on
	dialgebraic structures, see \cite[Chapter~13]{LodayVallette2012}.
	
	A \textbf{bar-unit} of a dialgebra \(A\) is an element \(\xi\in A\) satisfying
	\[
	\xi\vdash a=a=a\dashv \xi
	\qquad
	(\forall a\in A).
	\]
	The bar-unit need not be unique. The set of all bar-units is called the
	\textbf{halo} of \(A\) and will be denoted by
	\[
	\Halo(A).
	\]
	A bar-unit \(\xi\) is called \textbf{balanced} if
	\[
	a\vdash \xi=\xi\dashv a
	\qquad
	(\forall a\in A).
	\]
	
	Every dialgebra gives rise to a Leibniz bracket by
	\[
	[a,b]:=a\vdash b-b\dashv a.
	\]
	Indeed, the dialgebra identities imply the left Leibniz identity
	\[
	[a,[b,c]]=[[a,b],c]+[b,[a,c]]
	\qquad
	(\forall a,b,c\in A).
	\]
	Thus, if \(A\) is a dialgebra, the vector space \(A\) equipped with this
	bracket is a Leibniz algebra.
	
	\subsection{Racks, quandles and subracks}
	
	A \textbf{rack} is a nonempty set \(X\) endowed with a binary operation
	\[
	\triangleright:X\times X\longrightarrow X
	\]
	such that, for every \(x\in X\), the left translation
	\[
	L_x:X\longrightarrow X,
	\qquad
	L_x(y)=x\triangleright y,
	\]
	is bijective, and the self-distributivity identity
	\[
	x\triangleright(y\triangleright z)
	=
	(x\triangleright y)\triangleright(x\triangleright z)
	\]
	holds for all \(x,y,z\in X\). A rack is a \textbf{quandle} if, in addition,
	\[
	x\triangleright x=x
	\qquad
	(\forall x\in X).
	\]
	Racks and quandles appear naturally in knot theory and self-distributive
	algebra; see \cite[Section~2]{Joyce1982} and
	\cite[Section~1]{FennRourkeSanderson1995}.
	
	A \textbf{rack homomorphism} is a map \(f:X\to Y\) satisfying
	\[
	f(x\triangleright y)=f(x)\triangleright f(y)
	\qquad
	(\forall x,y\in X).
	\]
	A nonempty subset \(Y\subseteq X\) is called a \textbf{subrack} if it is closed
	under the rack operation and under the inverse left translations. In the finite
	case, closure under the rack operation already implies closure under inverse
	left translations. Indeed, if \(X\) is finite and \(Y\subseteq X\) is nonempty
	with
	\[
	x\triangleright y\in Y
	\qquad
	(\forall x,y\in Y),
	\]
	then each \(L_x\) restricts to an injective self-map of the finite set \(Y\),
	and hence to a bijection of \(Y\).
	
	\subsection{Generalized digroups and digroups}
	
	A \textbf{generalized digroup}, or briefly a \textbf{g-digroup}, is a nonempty
	set \(D\) endowed with two associative binary operations
	\[
	\vdash,\dashv:D\times D\longrightarrow D
	\]
	satisfying the mixed identities
	\[
	x\vdash(y\dashv z)=(x\vdash y)\dashv z,
	\]
	\[
	x\dashv(y\dashv z)=x\dashv(y\vdash z),
	\]
	and
	\[
	(x\vdash y)\vdash z=(x\dashv y)\vdash z
	\]
	for all \(x,y,z\in D\). Moreover, there exists at least one element
	\(\xi\in D\) such that
	\[
	x\dashv \xi=x=\xi\vdash x
	\qquad
	(\forall x\in D),
	\]
	and, for every bar-unit \(\xi\) and every \(x\in D\), there exist elements
	\[
	x^{-1}_{r,\xi},\ x^{-1}_{l,\xi}\in D
	\]
	such that
	\[
	x\vdash x^{-1}_{r,\xi}=\xi,
	\qquad
	x^{-1}_{l,\xi}\dashv x=\xi.
	\]
	The set of all bar-units of \(D\) is again called the \textbf{halo} and is
	denoted by
	\[
	\Halo(D).
	\]
	This notion was introduced in \cite[Section~2]{SVW}; related constructions and
	actions appear in \cite[Section~2]{RSV}.
	
	A \textbf{digroup} is a generalized digroup with a distinguished bar-unit
	\(\xi\) such that every \(x\in D\) admits a single inverse \(x^{-1}\) satisfying
	\[
	x\vdash x^{-1}=\xi=x^{-1}\dashv x.
	\]
	This is the balanced case relevant for the group-like elements of
	cocommutative Hopf dialgebras. Digroups and their connection with Leibniz
	algebras and Lie racks are discussed in \cite[Section~2]{Kinyon}.
	
	\subsection{Categorical conventions}
	
	We shall use the following categories. The category \(\mathbf{Rack}\) has racks
	as objects and rack homomorphisms as morphisms. The category \(\mathbf{gDig}\)
	has generalized digroups as objects and morphisms given by maps
	\[
	f:D\longrightarrow D'
	\]
	preserving both products and sending bar-units to bar-units, that is,
	\[
	f(x\vdash y)=f(x)\vdash f(y),\qquad
	f(x\dashv y)=f(x)\dashv f(y),
	\]
	and
	\[
	f(\Halo(D))\subseteq \Halo(D').
	\]
	The full subcategory of \(\mathbf{gDig}\) consisting of digroups with
	distinguished bar-unit-preserving morphisms will be denoted by \(\mathbf{Dig}\).
	
	We denote by \(\mathbf{HopfDialg}\) the category whose objects are
	cocommutative Hopf dialgebras and whose morphisms are coalgebra morphisms
	preserving the distinguished bar-unit, both dialgebra products and the
	antipode. The category \(\mathbf{RackBialg}\) consists of rack bialgebras and
	coalgebra morphisms preserving the distinguished group-like element and the
	rack product. The functorial language used below follows the standard
	conventions of category theory; see \cite[Chapter~I]{MacLane1998}.
	
	\section{Hopf dialgebras, generalized digroups and racks}
	
	This section develops the algebraic framework connecting cocommutative Hopf
	dialgebras, generalized digroups and racks. We first recall the adjoint rack
	bialgebra associated with a cocommutative Hopf dialgebra. We then review the
	structure theorem for generalized digroups and the construction of their
	conjugation racks. These ingredients will be used in the following sections to
	extract finite combinatorial invariants and to prove the group-like
	factorization theorem.
	
	\subsection{Cocommutative Hopf dialgebras and the associated rack functor}
	
	We recall the Hopf-dialgebraic framework used throughout the paper. Our
	conventions follow the cocommutative Hopf dialgebras of
	Alexandre--Bordemann--Rivi\`ere--Wagemann \cite{ABRW1}, written here with a
	distinguished bar-unit \(\xi\).
	
	\begin{definition}\label{def:hopf-dialgebra}
		Let \((A,\Delta,\varepsilon)\) be a cocommutative coalgebra over \(K\), let
		\[
		\overline{\xi}:K\longrightarrow A,
		\qquad
		\overline{\xi}(1_K)=\xi,
		\]
		and let
		\[
		\vdash,\dashv:A\otimes A\longrightarrow A
		\]
		be two \(K\)-linear associative maps. We say that
		\[
		(A,\Delta,\varepsilon,\xi,\vdash,\dashv)
		\]
		is a \textbf{cocommutative bar-unital di-coalgebra} if the following
		conditions hold.
		
		\begin{enumerate}
			\item The products \(\vdash\) and \(\dashv\) satisfy the dialgebra
			identities
			\[
			(a\vdash b)\vdash c=(a\dashv b)\vdash c,
			\]
			\[
			a\dashv(b\dashv c)=a\dashv(b\vdash c),
			\]
			and
			\[
			(a\vdash b)\dashv c=a\vdash(b\dashv c)
			\]
			for all \(a,b,c\in A\).
			
			\item The element \(\xi\) is a balanced bar-unit, that is,
			\[
			\xi\vdash a=a=a\dashv \xi,
			\qquad
			a\vdash \xi=\xi\dashv a
			\qquad
			(\forall a\in A).
			\]
			
			\item The element \(\xi\) is group-like:
			\[
			\Delta(\xi)=\xi\otimes \xi,
			\qquad
			\varepsilon(\xi)=1.
			\]
			
			\item Both products are coalgebra morphisms. More precisely, we endow
			\(A\otimes A\) with the tensor-product coalgebra structure
			\[
			\Delta_{A\otimes A}
			:=
			(\id_A\otimes\tau_{A,A}\otimes\id_A)
			\circ
			(\Delta_A\otimes\Delta_A),
			\qquad
			\varepsilon_{A\otimes A}
			:=
			m_K\circ(\varepsilon_A\otimes\varepsilon_A),
			\]
			where
			\[
			\tau_{A,A}:A\otimes A\longrightarrow A\otimes A,
			\qquad
			\tau_{A,A}(u\otimes v)=v\otimes u,
			\]
			and
			\[
			m_K:K\otimes K\longrightarrow K,
			\qquad
			m_K(\lambda\otimes\mu)=\lambda\mu.
			\]
			Then, for each product
			\[
			\star\in\{\vdash,\dashv\},
			\]
			the map
			\[
			\star:A\otimes A\longrightarrow A
			\]
			is required to be a coalgebra morphism, that is, for each \(\star\in\{\vdash,\dashv\}\), the diagrams
			\[
			\xymatrix{
				A\otimes A \ar[r]^{\star} \ar[d]_{\Delta_{A\otimes A}}
				& A \ar[d]^{\Delta_A} \\
				(A\otimes A)\otimes(A\otimes A) \ar[r]_{\star\otimes\star}
				& A\otimes A
			}
			\]
			and
			\[
			\xymatrix{
				A\otimes A \ar[r]^{\star} \ar[d]_{\varepsilon_{A\otimes A}}
				& A \ar[d]^{\varepsilon_A} \\
				K \ar[r]_{\id_K}
				& K
			}
			\]
			commute.
			
			Equivalently, the following identities hold:
			\[
			\Delta_A\circ\star
			=
			(\star\otimes\star)
			\circ
			(\id_A\otimes\tau_{A,A}\otimes\id_A)
			\circ
			(\Delta_A\otimes\Delta_A),
			\]
			and
			\[
			\varepsilon_A\circ\star
			=
			m_K\circ(\varepsilon_A\otimes\varepsilon_A).
			\]
			In Sweedler notation, this means that, for all \(a,b\in A\),
			\[
			\Delta_A(a\vdash b)
			=
			\sum_{(a),(b)}
			(a_{(1)}\vdash b_{(1)})
			\otimes
			(a_{(2)}\vdash b_{(2)}),
			\qquad
			\varepsilon_A(a\vdash b)
			=
			\varepsilon_A(a)\varepsilon_A(b),
			\]
			and
			\[
			\Delta_A(a\dashv b)
			=
			\sum_{(a),(b)}
			(a_{(1)}\dashv b_{(1)})
			\otimes
			(a_{(2)}\dashv b_{(2)}),
			\qquad
			\varepsilon_A(a\dashv b)
			=
			\varepsilon_A(a)\varepsilon_A(b).
			\]
		\end{enumerate}
		
		If, moreover, there exists a coalgebra morphism
		\[
		S:A\longrightarrow A,
		\]
		called the common antipode, such that
		\[
		\sum_{(a)}a_{(1)}\vdash S(a_{(2)})=\varepsilon(a)\xi,
		\qquad
		\sum_{(a)}S(a_{(1)})\dashv a_{(2)}=\varepsilon(a)\xi
		\]
		for every \(a\in A\), then
		\[
		(A,\Delta,\varepsilon,\xi,\vdash,\dashv,S)
		\]
		is called a \textbf{cocommutative Hopf dialgebra}.
	\end{definition}
	
	\begin{remark}\label{rem:balanced-antipode}
		In the general Hopf dialgebra setting, one may distinguish a right antipode
		for \(\vdash\) and a left antipode for \(\dashv\). In the balanced situation
		considered here these two antipodes are represented by a single coalgebra
		morphism \(S\). This one-antipode hypothesis is precisely what later forces
		the group-like elements to form a digroup rather than only a generalized
		digroup.
	\end{remark}
	
	The primitive elements of a cocommutative Hopf dialgebra carry a Leibniz
	algebra structure. Although this fact will not be the main focus of the paper,
	it explains the link with the Leibniz-theoretic origin of Hopf dialgebras.
	
	\begin{proposition}[cf. {\cite[Proposition 2.2]{ABRW1}}]\label{prop:primitive-leibniz}
		Let
		\[
		A=(A,\Delta,\varepsilon,\xi,\vdash,\dashv,S)
		\]
		be a cocommutative Hopf dialgebra. Then
		\[
		\Prim(A):=\{x\in A\mid \Delta(x)=x\otimes\xi+\xi\otimes x\}
		\]
		is closed under the bracket
		\[
		[x,y]:=x\vdash y-y\dashv x.
		\]
		Hence \(\Prim(A)\) is a Leibniz algebra.
	\end{proposition}
	
	\begin{proof}
		Let \(x,y\in\Prim(A)\). Since \(\vdash\) and \(\dashv\) are coalgebra
		morphisms, expanding \(\Delta(x\vdash y)\) and \(\Delta(y\dashv x)\) gives
		\begin{align*}
			\Delta(x\vdash y)
			&=
			(x\vdash y)\otimes\xi
			+
			(x\vdash\xi)\otimes y
			+
			y\otimes(x\vdash\xi)
			+
			\xi\otimes(x\vdash y),
			\\
			\Delta(y\dashv x)
			&=
			(y\dashv x)\otimes\xi
			+
			y\otimes(\xi\dashv x)
			+
			(\xi\dashv x)\otimes y
			+
			\xi\otimes(y\dashv x).
		\end{align*}
		Using the balanced bar-unit relation
		\[
		x\vdash\xi=\xi\dashv x,
		\]
		the middle terms cancel in the difference. Hence
		\[
		\Delta(x\vdash y-y\dashv x)
		=
		(x\vdash y-y\dashv x)\otimes\xi
		+
		\xi\otimes(x\vdash y-y\dashv x).
		\]
		Thus \(x\vdash y-y\dashv x\in\Prim(A)\).
		
		Finally, for every associative dialgebra, the bracket
		\[
		[x,y]=x\vdash y-y\dashv x
		\]
		satisfies the left Leibniz identity as a formal consequence of the three
		mixed dialgebra identities. Hence the restriction of this bracket to the
		closed subspace \(\Prim(A)\) is again Leibniz.
	\end{proof}
	
	We now recall the rack bialgebra associated with a cocommutative Hopf
	dialgebra.
	
	\begin{definition}[{\cite[Definition~2.1]{ABRW1}}]\label{def:rack-bialgebra}
		Let \((B,\Delta,\varepsilon)\) be a cocommutative coalgebra over \(K\), and
		let \(\xi\in B\) be group-like. A \(K\)-linear map
		\[
		\triangleright:B\otimes B\longrightarrow B
		\]
		makes
		\[
		(B,\Delta,\varepsilon,\xi,\triangleright)
		\]
		a \textbf{cocommutative rack bialgebra} if:
		\begin{enumerate}
			\item \(\triangleright\) is a coalgebra morphism;
			\item \(\xi\triangleright a=a\) for all \(a\in B\);
			\item \(a\triangleright\xi=\varepsilon(a)\xi\) for all \(a\in B\);
			\item the self-distributivity identity
			\[
			a\triangleright(b\triangleright c)
			=
			\sum_{(a)}
			(a_{(1)}\triangleright b)\triangleright
			(a_{(2)}\triangleright c)
			\]
			holds for all \(a,b,c\in B\).
		\end{enumerate}
	\end{definition}
	
	\begin{proposition}\label{prop:adjoint-rack-bialgebra}
		Let
		\[
		A=(A,\Delta,\varepsilon,\xi,\vdash,\dashv,S)
		\]
		be a cocommutative Hopf dialgebra. Define
		\[
		a\triangleright b
		:=
		\sum_{(a)}(a_{(1)}\vdash b)\dashv S(a_{(2)}).
		\]
		Then
		\[
		\Adj(A):=(A,\Delta,\varepsilon,\xi,\triangleright)
		\]
		is a cocommutative rack bialgebra.
	\end{proposition}
	
	\begin{proof}
		This is \cite[Proposition~2.6]{ABRW1}, written in the present notation.
	\end{proof}
	
	For a rack bialgebra \(B\), define its set of set-like elements by
	\[
	\Slike(B):=\{b\in B\mid \Delta(b)=b\otimes b,\ \varepsilon(b)=1\}.
	\]
	Since the rack product is a coalgebra morphism, \(\Slike(B)\) is closed under
	\(\triangleright\).
	
	In general, closure of set-like elements under the rack product does not by
	itself guarantee that the induced left translations are bijective. In this
	paper we only use this construction for adjoint rack bialgebras arising from
	cocommutative Hopf dialgebras. In that case, the antipode provides the inverse
	of the corresponding left translations on group-like elements, and the set-like
	elements form a genuine rack.
	
	Thus, for a cocommutative Hopf dialgebra \(A\), we define
	\[
	\mathcal R(A)
	:=
	\Slike(\Adj(A)).
	\]
	On morphisms, \(\mathcal R\) is obtained by restricting the corresponding
	coalgebra morphism to set-like elements. Hence we obtain a functor
	\[
	\mathcal R:
	\mathbf{HopfDialg}
	\longrightarrow
	\mathbf{Rack}.
	\]
	The main point of Theorem~\ref{thm:factorization-through-glike} is that this
	functor factorizes through the digroup of group-like elements.
	
	\subsection{Generalized digroups and conjugation racks}
	
	We now recall the structural decomposition of generalized digroups and the
	associated conjugation rack construction. The basic definitions and categorical
	conventions were fixed in Section~\ref{sec:preliminaries}. The main point is
	that, after choosing a bar-unit, every generalized digroup is controlled by an
	ordinary group acting on its halo.
	
	The following structure theorem is the main structural tool for generalized
	digroups. It shows that, after choosing a bar-unit, every generalized digroup
	is controlled by an ordinary group acting on its halo.
	
	\begin{theorem}\label{thm:gdigroup-structure}
		Let \(D\) be a generalized digroup and fix a bar-unit
		\[
		\xi\in\Halo(D).
		\]
		Let \(G_l^\xi\) denote the left group associated with \(\xi\). Then:
		\begin{enumerate}
			\item \(G_l^\xi\) is a group with respect to \(\dashv\), with identity
			\(\xi\);
			
			\item the halo \(\Halo(D)\) is a \(G_l^\xi\)-set for the action
			\[
			a\bullet_l \eta
			:=
			a\vdash \eta\dashv a^{-1}
			\qquad
			(a\in G_l^\xi,\ \eta\in\Halo(D)),
			\]
			where \(a^{-1}\) denotes the inverse of \(a\) in the group
			\(G_l^\xi\);
			
			\item writing the group law of \(G_l^\xi\) multiplicatively, the set
			\[
			G_l^\xi\times \Halo(D)
			\]
			becomes a generalized digroup with products
			\[
			(a,\alpha)\vdash(b,\beta)
			=
			\bigl(ab,a\bullet_l\beta\bigr),
			\qquad
			(a,\alpha)\dashv(b,\beta)
			=
			\bigl(ab,\alpha\bigr);
			\]
			
			\item the map
			\[
			\varphi_l:D\longrightarrow G_l^\xi\times\Halo(D),
			\qquad
			\varphi_l(x)
			=
			\bigl(\xi\dashv x,\ x\dashv x^{-1}_{l,\xi}\bigr)
			\]
			is an isomorphism of generalized digroups.
		\end{enumerate}
	\end{theorem}
	
	\begin{proof}
		This is the decomposition theorem for generalized digroups. We refer to
		\cite[Theorem~5]{SVW} and to the exposition in \cite[Theorem~2]{RSV4}.
	\end{proof}
	
	\begin{notation}\label{not:g-cross-e}
		After choosing a bar-unit \(\xi\), we shall often use the structural
		identification
		\[
		D\simeq G\times E,
		\]
		where
		\[
		G:=G_l^\xi,
		\qquad
		E:=\Halo(D).
		\]
		We write the action simply as
		\[
		g\bullet\alpha
		\qquad
		(g\in G,\ \alpha\in E),
		\]
		and the two products on \(G\times E\) as
		\[
		(g,\alpha)\vdash(h,\beta)
		=
		(gh,g\bullet\beta),
		\qquad
		(g,\alpha)\dashv(h,\beta)
		=
		(gh,\alpha).
		\]
		The bar-units in this model are precisely the elements of the form
		\[
		(e,\alpha),
		\qquad
		\alpha\in E,
		\]
		where \(e\) is the identity of \(G\).
	\end{notation}
	
	\begin{remark}\label{rem:choice-of-bar-unit}
		The structural presentation \(D\simeq G\times E\) depends on the chosen
		bar-unit \(\xi\). However, the conjugation rack \(\Conj(D)\) is intrinsic:
		different choices of bar-unit give isomorphic descriptions of the same rack.
		In what follows, the chosen presentation is used only as a coordinate model
		for computing the rack operation and its finite combinatorial invariants.
	\end{remark}
	
	We now pass from generalized digroups to racks. The essential point is that
	the conjugation formula is independent of the one-sided inverse chosen.
	
	\begin{proposition}[Conjugation rack of a generalized digroup]\label{prop:conj-rack-gdigroup}
		Let \(D\) be a generalized digroup and fix a bar-unit
		\(\xi\in\Halo(D)\). For \(x,y\in D\), define
		\[
		x\triangleright y
		:=
		(x\vdash y)\dashv x^{-1},
		\]
		where \(x^{-1}\) denotes any one-sided inverse of \(x\) relative to the chosen
		bar-unit, either a right inverse for \(\vdash\) or a left inverse for \(\dashv\).
		Then this operation is independent of the chosen inverse and defines a rack
		structure on \(D\). We denote this rack by
		\[
		\Conj(D).
		\]
		Moreover, the assignment \(D\mapsto \Conj(D)\) is functorial:
		every homomorphism of generalized digroups
		\[
		f:D\longrightarrow D'
		\]
		induces a rack homomorphism
		\[
		\Conj(f):\Conj(D)\longrightarrow\Conj(D').
		\]
		Thus there is a functor
		\[
		\Conj:\mathbf{gDig}\longrightarrow\mathbf{Rack}.
		\]
	\end{proposition}
	
	\begin{proof}
		Using Theorem~\ref{thm:gdigroup-structure}, identify
		\[
		D\simeq G\times E
		\]
		as in Notation~\ref{not:g-cross-e}. Thus
		\[
		(g,\alpha)\vdash(h,\beta)
		=
		(gh,g\bullet\beta),
		\qquad
		(g,\alpha)\dashv(h,\beta)
		=
		(gh,\alpha).
		\]
		
		Fix a bar-unit \((e,\eta)\). A right inverse of \((g,\alpha)\) relative to
		\((e,\eta)\) is
		\[
		(g^{-1},g^{-1}\bullet\eta),
		\]
		because
		\[
		(g,\alpha)\vdash(g^{-1},g^{-1}\bullet\eta)
		=
		(e,g\bullet(g^{-1}\bullet\eta))
		=
		(e,\eta).
		\]
		A left inverse of \((g,\alpha)\) relative to \((e,\eta)\) is
		\[
		(g^{-1},\eta),
		\]
		because
		\[
		(g^{-1},\eta)\dashv(g,\alpha)
		=
		(e,\eta).
		\]
		
		Now let \((h,\beta)\in G\times E\). If one uses the right inverse, then
		\begin{align*}
			\bigl((g,\alpha)\vdash(h,\beta)\bigr)
			\dashv
			(g^{-1},g^{-1}\bullet\eta)
			&=
			(gh,g\bullet\beta)\dashv(g^{-1},g^{-1}\bullet\eta)\\
			&=
			(ghg^{-1},g\bullet\beta).
		\end{align*}
		If one uses the left inverse, then
		\begin{align*}
			\bigl((g,\alpha)\vdash(h,\beta)\bigr)
			\dashv
			(g^{-1},\eta)
			&=
			(gh,g\bullet\beta)\dashv(g^{-1},\eta)\\
			&=
			(ghg^{-1},g\bullet\beta).
		\end{align*}
		Hence the conjugation operation is independent of the chosen one-sided
		inverse. In the model \(G\times E\), it is given by the explicit formula
		\begin{equation}\label{eq:conj-rack-formula}
			(g,\alpha)\triangleright(h,\beta)
			=
			(ghg^{-1},g\bullet\beta).
		\end{equation}
		
		We now verify the rack axioms. For fixed \((g,\alpha)\), the left
		translation is
		\[
		L_{(g,\alpha)}(h,\beta)
		=
		(ghg^{-1},g\bullet\beta).
		\]
		This map is bijective, with inverse
		\[
		(h,\beta)
		\longmapsto
		(g^{-1}hg,g^{-1}\bullet\beta).
		\]
		Thus all left translations are bijections.
		
		It remains to prove self-distributivity. Let
		\[
		(g,\alpha),(h,\beta),(k,\gamma)\in G\times E.
		\]
		Using \eqref{eq:conj-rack-formula}, we compute
		\begin{align*}
			(g,\alpha)\triangleright
			\bigl((h,\beta)\triangleright(k,\gamma)\bigr)
			&=
			(g,\alpha)\triangleright(hkh^{-1},h\bullet\gamma)\\
			&=
			(ghkh^{-1}g^{-1},g\bullet(h\bullet\gamma)).
		\end{align*}
		On the other hand,
		\begin{align*}
			\bigl((g,\alpha)\triangleright(h,\beta)\bigr)
			\triangleright
			\bigl((g,\alpha)\triangleright(k,\gamma)\bigr)
			&=
			(ghg^{-1},g\bullet\beta)
			\triangleright
			(gkg^{-1},g\bullet\gamma)\\
			&=
			\bigl((ghg^{-1})(gkg^{-1})(ghg^{-1})^{-1},
			(ghg^{-1})\bullet(g\bullet\gamma)\bigr)\\
			&=
			(ghkh^{-1}g^{-1},g\bullet(h\bullet\gamma)).
		\end{align*}
		Therefore
		\[
		(g,\alpha)\triangleright
		\bigl((h,\beta)\triangleright(k,\gamma)\bigr)
		=
		\bigl((g,\alpha)\triangleright(h,\beta)\bigr)
		\triangleright
		\bigl((g,\alpha)\triangleright(k,\gamma)\bigr).
		\]
		Hence \(\Conj(D)\) is a rack.
		
		Finally, let \(f:D\to D'\) be a morphism of generalized digroups in the above
		sense. Thus \(f\) preserves both products and sends bar-units to bar-units.
		If \(x^{-1}\) is a one-sided inverse of \(x\) relative to a bar-unit \(\xi\),
		then \(f(x^{-1})\) is the corresponding one-sided inverse of \(f(x)\) relative
		to the bar-unit \(f(\xi)\). Thus
		\[
		f(x\triangleright y)
		=
		f\bigl((x\vdash y)\dashv x^{-1}\bigr)
		=
		(f(x)\vdash f(y))\dashv f(x^{-1})
		=
		f(x)\triangleright f(y).
		\]
		Therefore \(f\) is a rack homomorphism between the associated conjugation
		racks. Identities and compositions are inherited from generalized digroup
		homomorphisms, so \(D\mapsto\Conj(D)\) defines a functor.
	\end{proof}
	
	\begin{corollary}[Quandle criterion]\label{cor:quandle-criterion}
		Let \(D\) be a generalized digroup and choose a structural identification
		\[
		D\simeq G\times E
		\]
		as above. Then \(\Conj(D)\) is a quandle if and only if the action of
		\(G\) on \(E\) is trivial.
	\end{corollary}
	
	\begin{proof}
		By \eqref{eq:conj-rack-formula},
		\[
		(g,\alpha)\triangleright(g,\alpha)
		=
		(ggg^{-1},g\bullet\alpha)
		=
		(g,g\bullet\alpha).
		\]
		Thus \(\Conj(D)\) is idempotent if and only if
		\[
		g\bullet\alpha=\alpha
		\qquad
		(\forall g\in G,\ \forall \alpha\in E).
		\]
		This is precisely the condition that the action of \(G\) on \(E\) be
		trivial.
	\end{proof}
	
	\begin{remark}\label{rem:conj-rack-base-fiber}
		Formula \eqref{eq:conj-rack-formula} shows that the conjugation rack of a
		generalized digroup is a fibered version of the ordinary conjugation rack
		of \(G\). The first coordinate is transformed by group conjugation, while
		the second coordinate is transformed by the given action of \(G\) on the
		halo \(E\).
	\end{remark}
	
	Formula~\eqref{eq:conj-rack-formula} reduces the conjugation rack of a
	generalized digroup to a group action on its halo. We now exploit this
	description in the finite case to extract explicit combinatorial invariants.
	
	\section{Finite rack combinatorics of generalized digroups}
	
	We now isolate the finite combinatorial structure carried by the conjugation
	rack of a generalized digroup. Throughout this section, \(D\) denotes a
	finite generalized digroup. After choosing a bar-unit
	\[
	\xi\in\Halo(D)
	\]
	and applying Theorem~\ref{thm:gdigroup-structure}, we identify
	\[
	D\simeq G\times E,
	\]
	where
	\[
	G:=G_l^\xi
	\qquad\text{and}\qquad
	E:=\Halo(D).
	\]
	Thus \(G\) is a finite group and \(E\) is a finite \(G\)-set. We write
	\[
	\rho:G\longrightarrow \operatorname{Sym}(E),
	\qquad
	\rho(g)(\beta)=g\bullet\beta,
	\]
	for the action of \(G\) on \(E\), and
	\[
	c_g:G\longrightarrow G,
	\qquad
	c_g(h)=ghg^{-1},
	\]
	for the conjugation permutation of \(G\).
	
	The first result records the structural form of the conjugation rack and its
	inner permutation group.
	
	\begin{theorem}[Structure theorem for the conjugation rack]\label{thm:conj-structure}
		Let \(D\) be a finite generalized digroup and identify
		\[
		D\simeq G\times E
		\]
		as above. Then the conjugation rack of \(D\) is the rack on \(G\times E\)
		given by
		\[
		(g,\alpha)\triangleright(h,\beta)
		=
		(ghg^{-1},g\bullet\beta).
		\]
		Moreover:
		\begin{enumerate}
			\item the projection
			\[
			\pi:\Conj(D)\longrightarrow\Conj(G),
			\qquad
			\pi(g,\alpha)=g,
			\]
			is a surjective rack morphism;
			
			\item for every \((g,\alpha)\in G\times E\), the left translation is
			\[
			L_{(g,\alpha)}=c_g\times \rho(g),
			\]
			hence it depends only on the first coordinate \(g\);
			
			\item the inner permutation group of \(\Conj(D)\) is
			\[
			\operatorname{Inn}(\Conj(D))
			\cong
			G/\bigl(Z(G)\cap \operatorname{Ker}(\rho)\bigr).
			\]
		\end{enumerate}
		In particular, \(\Conj(D)\) is a rack over the ordinary conjugation rack
		\(\Conj(G)\), with fiber \(E\).
	\end{theorem}
	
	\begin{proof}
		The formula
		\[
		(g,\alpha)\triangleright(h,\beta)
		=
		(ghg^{-1},g\bullet\beta)
		\]
		was obtained in Proposition~\ref{prop:conj-rack-gdigroup}. Therefore
		\[
		\pi\bigl((g,\alpha)\triangleright(h,\beta)\bigr)
		=
		ghg^{-1}
		=
		\pi(g,\alpha)\triangleright \pi(h,\beta),
		\]
		so \(\pi\) is a rack morphism. It is clearly surjective.
		
		For the second assertion, the left translation by \((g,\alpha)\) is
		\[
		L_{(g,\alpha)}(h,\beta)
		=
		(g,\alpha)\triangleright(h,\beta)
		=
		(ghg^{-1},g\bullet\beta).
		\]
		This is exactly the product permutation
		\[
		c_g\times\rho(g)
		\]
		on \(G\times E\). Hence \(L_{(g,\alpha)}\) depends only on \(g\).
		
		Define
		\[
		\lambda:G\longrightarrow \operatorname{Inn}(\Conj(D)),
		\qquad
		\lambda(g)=c_g\times\rho(g).
		\]
		Equivalently, \(\lambda(g)=L_{(g,\alpha)}\) for any choice of
		\(\alpha\in E\). Since
		\[
		c_g\circ c_h=c_{gh}
		\qquad\text{and}\qquad
		\rho(g)\circ\rho(h)=\rho(gh),
		\]
		we have
		\[
		\lambda(g)\lambda(h)=\lambda(gh),
		\]
		so \(\lambda\) is a group homomorphism. Its image is the whole inner group,
		because \(\operatorname{Inn}(\Conj(D))\) is generated by all left
		translations and each left translation has the form \(c_g\times\rho(g)\).
		
		Now \(g\in\operatorname{Ker}(\lambda)\) if and only if
		\[
		c_g(h)=h
		\qquad
		(\forall h\in G)
		\]
		and
		\[
		\rho(g)(\beta)=\beta
		\qquad
		(\forall \beta\in E).
		\]
		The first condition is equivalent to \(g\in Z(G)\), while the second is
		equivalent to \(g\in\operatorname{Ker}(\rho)\). Therefore
		\[
		\operatorname{Ker}(\lambda)
		=
		Z(G)\cap\operatorname{Ker}(\rho).
		\]
		The first isomorphism theorem gives
		\[
		\operatorname{Inn}(\Conj(D))
		\cong
		G/\bigl(Z(G)\cap\operatorname{Ker}(\rho)\bigr).
		\]
	\end{proof}
	
	We next introduce a cycle-index-type invariant for finite racks. The terminology
	is inspired by the classical cycle-index methods in enumerative combinatorics
	\cite{Read1968,HararyPalmer}. Since the average is taken over the left
	translations of the rack, rather than over a permutation group, we use the term
	\textbf{left-translation cycle index}.
	
	\begin{definition}\label{def:left-translation-cycle-index}
		Let \(X\) be a finite rack. For a permutation \(\sigma\) of \(X\), let
		\(c_\ell(\sigma)\) denote the number of cycles of length \(\ell\) in the
		cycle decomposition of \(\sigma\). The \textbf{left-translation cycle index}
		of \(X\) is
		\[
		Z_X^{\mathrm{LT}}(u_1,u_2,\ldots)
		:=
		\frac{1}{|X|}
		\sum_{x\in X}
		\prod_{\ell\geq 1}
		u_\ell^{\,c_\ell(L_x)}.
		\]
		Since \(X\) is finite, only finitely many exponents in each monomial are
		nonzero.
	\end{definition}
	
	\begin{theorem}[Left-translation cycle-index theorem]\label{thm:cycle-index}
		Let \(D\) be a finite generalized digroup and write
		\[
		D\simeq G\times E
		\]
		as above. For each \(g\in G\), let
		\[
		a_i(g):=c_i(c_g),
		\qquad
		b_j(g):=c_j(\rho(g)).
		\]
		Then, for every \((g,\alpha)\in G\times E\), the number of
		\(\ell\)-cycles of the left translation \(L_{(g,\alpha)}\) is
		\[
		c_\ell\bigl(L_{(g,\alpha)}\bigr)
		=
		\sum_{\substack{i,j\geq 1\\ \operatorname{lcm}(i,j)=\ell}}
		a_i(g)b_j(g)\gcd(i,j).
		\]
		Consequently,
		\[
		Z_{\Conj(D)}^{\mathrm{LT}}(u_1,u_2,\ldots)
		=
		\frac{1}{|G|}
		\sum_{g\in G}
		\prod_{\ell\geq 1}
		u_\ell^{
			\sum_{\substack{i,j\geq 1\\ \operatorname{lcm}(i,j)=\ell}}
			a_i(g)b_j(g)\gcd(i,j)
		}.
		\]
	\end{theorem}
	
	\begin{proof}
		By Theorem~\ref{thm:conj-structure}, for every \((g,\alpha)\in G\times E\),
		\[
		L_{(g,\alpha)}=c_g\times\rho(g).
		\]
		Fix \(g\in G\). Let \(C\) be an \(i\)-cycle of \(c_g\), and let \(F\) be a
		\(j\)-cycle of \(\rho(g)\). The product permutation
		\[
		c_g\times\rho(g)
		\]
		restricts to a permutation of \(C\times F\), which has \(ij\) elements.
		
		Starting from \((x,y)\in C\times F\), after \(m\) iterations one obtains
		\[
		(c_g\times\rho(g))^m(x,y)
		=
		(c_g^m(x),\rho(g)^m(y)).
		\]
		Hence \((x,y)\) returns to itself if and only if \(m\) is simultaneously a
		multiple of \(i\) and of \(j\). Therefore every orbit in \(C\times F\) has
		length
		\[
		\operatorname{lcm}(i,j).
		\]
		Since \(C\times F\) has \(ij\) elements, the number of such cycles is
		\[
		\frac{ij}{\operatorname{lcm}(i,j)}
		=
		\gcd(i,j).
		\]
		Thus each pair consisting of an \(i\)-cycle of \(c_g\) and a \(j\)-cycle of
		\(\rho(g)\) contributes exactly \(\gcd(i,j)\) cycles of length
		\(\operatorname{lcm}(i,j)\). Summing over all such cycles gives
		\[
		c_\ell\bigl(L_{(g,\alpha)}\bigr)
		=
		\sum_{\substack{i,j\geq 1\\ \operatorname{lcm}(i,j)=\ell}}
		a_i(g)b_j(g)\gcd(i,j).
		\]
		
		Now \(L_{(g,\alpha)}\) depends only on \(g\). For each fixed \(g\), there are
		exactly \(|E|\) elements of the form \((g,\alpha)\). Since
		\[
		|D|=|G||E|,
		\]
		the average in Definition~\ref{def:left-translation-cycle-index} becomes
		\[
		\frac{1}{|G||E|}
		\sum_{g\in G}
		|E|
		\prod_{\ell\geq 1}
		u_\ell^{
			\sum_{\substack{i,j\geq 1\\ \operatorname{lcm}(i,j)=\ell}}
			a_i(g)b_j(g)\gcd(i,j)
		}.
		\]
		This simplifies to the stated formula.
	\end{proof}
	
	A useful specialization is obtained by keeping only the number of fixed points
	of each left translation.
	
	\begin{corollary}[Fixed-point polynomial]\label{cor:fixed-point-polynomial}
		Let
		\[
		\operatorname{Fix}_E(g)
		:=
		\{\beta\in E\mid g\bullet\beta=\beta\}
		\]
		and let
		\[
		C_G(g):=\{h\in G\mid gh=hg\}
		\]
		be the centralizer of \(g\) in \(G\). Define
		\[
		\Phi_D(t)
		:=
		\sum_{x\in D}
		t^{\lvert \operatorname{Fix}(L_x)\rvert }.
		\]
		Then
		\[
		\Phi_D(t)
		=
		\lvert E\rvert 
		\sum_{g\in G}
		t^{\,\lvert C_G(g)\vert \,\lvert \operatorname{Fix}_E(g)\rvert }.
		\]
	\end{corollary}
	
	\begin{proof}
		By Theorem~\ref{thm:conj-structure},
		\[
		L_{(g,\alpha)}(h,\beta)=(h,\beta)
		\]
		if and only if
		\[
		ghg^{-1}=h
		\qquad\text{and}\qquad
		g\bullet\beta=\beta.
		\]
		The first condition is equivalent to \(h\in C_G(g)\), and the second to
		\(\beta\in\operatorname{Fix}_E(g)\). Hence
		\[
		\lvert \operatorname{Fix}(L_{(g,\alpha)})\rvert
		=
		\lvert C_G(g)\rvert \,\lvert \operatorname{Fix}_E(g)\rvert.
		\]
		Since \(L_{(g,\alpha)}\) depends only on \(g\), and since there are \(|E|\)
		elements of the form \((g,\alpha)\), summing over all \(g\in G\) gives the
		claim.
	\end{proof}
	
	The same fixed-point data also determines the number of orbits of the inner
	permutation group.
	
	\begin{corollary}[Orbit formula]\label{cor:orbit-formula}
		Let \(D\) be a finite generalized digroup and write \(D\simeq G\times E\)
		as above. Then the number of orbits of \(\operatorname{Inn}(\Conj(D))\) on
		\(\Conj(D)\) is
		\[
		\#\bigl((G\times E)/\operatorname{Inn}(\Conj(D))\bigr)
		=
		\frac{1}{\lvert G\rvert }
		\sum_{g\in G}
		\lvert C_G(g)\rvert \,\lvert \operatorname{Fix}_E(g)\rvert .
		\]
	\end{corollary}
	
	\begin{proof}
		The inner group action on \(G\times E\) is induced by the homomorphism
		\[
		G\longrightarrow \operatorname{Sym}(G\times E),
		\qquad
		g\longmapsto c_g\times\rho(g).
		\]
		Although this homomorphism may have a nontrivial kernel, Burnside's lemma
		applied to \(G\) gives the correct orbit count, because the \(G\)-orbits are
		precisely the \(\operatorname{Inn}(\Conj(D))\)-orbits. Thus
		\[
		\#\bigl((G\times E)/\operatorname{Inn}(\Conj(D))\bigr)
		=
		\frac{1}{|G|}
		\sum_{g\in G}
		\lvert \operatorname{Fix}(c_g\times\rho(g))\rvert.
		\]
		Now
		\[
		\lvert\operatorname{Fix}(c_g\times\rho(g))\rvert
		=
		\lvert C_G(g)\rvert\,\lvert\operatorname{Fix}_E(g)\rvert.
		\]
		The formula follows.
	\end{proof}
	
	We now turn to subracks. Throughout the rest of this section, subracks are
	assumed to be nonempty.
	
	\begin{theorem}[Subrack structure theorem]\label{thm:subrack-structure}
		Let \(D\) be a finite generalized digroup and write
		\[
		D\simeq G\times E
		\]
		as above. Let \(X\subseteq G\times E\) be a nonempty subset, and define
		\[
		Y:=\pi(X)\subseteq G,
		\qquad
		X_g:=\{\beta\in E\mid (g,\beta)\in X\}
		\qquad
		(g\in Y).
		\]
		Then \(X\) is a subrack of \(\Conj(D)\) if and only if the following
		conditions hold:
		\begin{enumerate}
			\item \(Y\) is a subrack of \(\Conj(G)\);
			\item for every \(g,h\in Y\),
			\[
			g\bullet X_h\subseteq X_{ghg^{-1}}.
			\]
		\end{enumerate}
	\end{theorem}
	
	\begin{proof}
		Assume first that \(X\) is a subrack of \(\Conj(D)\). Since
		\[
		\pi:\Conj(D)\longrightarrow\Conj(G)
		\]
		is a rack morphism by Theorem~\ref{thm:conj-structure}, the image
		\[
		Y=\pi(X)
		\]
		is closed under the rack operation of \(\Conj(G)\). Since \(X\) is
		nonempty, \(Y\) is nonempty. Thus \(Y\) is a subrack of \(\Conj(G)\).
		
		Now let \(g,h\in Y\), and let \(\beta\in X_h\). By definition,
		\[
		(h,\beta)\in X.
		\]
		Since \(g\in Y\), there exists \(\alpha\in E\) such that
		\[
		(g,\alpha)\in X.
		\]
		Because \(X\) is closed under the rack operation,
		\[
		(g,\alpha)\triangleright(h,\beta)
		=
		(ghg^{-1},g\bullet\beta)
		\in X.
		\]
		Therefore
		\[
		g\bullet\beta\in X_{ghg^{-1}},
		\]
		and hence
		\[
		g\bullet X_h\subseteq X_{ghg^{-1}}.
		\]
		
		Conversely, assume that (1) and (2) hold. Let
		\[
		(g,\alpha),(h,\beta)\in X.
		\]
		Then \(g,h\in Y\) and \(\beta\in X_h\). By condition (2),
		\[
		g\bullet\beta\in X_{ghg^{-1}}.
		\]
		Equivalently,
		\[
		(ghg^{-1},g\bullet\beta)\in X.
		\]
		Since
		\[
		(g,\alpha)\triangleright(h,\beta)
		=
		(ghg^{-1},g\bullet\beta),
		\]
		it follows that \(X\) is closed under the rack operation. Therefore \(X\)
		is a subrack of \(\Conj(D)\).
	\end{proof}
	
	The previous theorem immediately produces a useful family of product-type
	subracks.
	
	\begin{corollary}[Product-type subracks]\label{cor:product-subracks}
		Let \(Y\) be a subrack of \(\Conj(G)\), and let \(F\subseteq E\) be a
		nonempty subset stable under the action of the subgroup generated by \(Y\)
		inside \(G\). Then
		\[
		Y\times F
		\]
		is a subrack of \(\Conj(D)\).
	\end{corollary}
	
	\begin{proof}
		For every \(h\in Y\), one has
		\[
		(Y\times F)_h=F.
		\]
		If \(g,h\in Y\), then
		\[
		ghg^{-1}\in Y
		\]
		because \(Y\) is a subrack of \(\Conj(G)\). Moreover, since \(F\) is stable
		under the subgroup generated by \(Y\),
		\[
		g\bullet F\subseteq F.
		\]
		Thus
		\[
		g\bullet(Y\times F)_h
		=
		g\bullet F
		\subseteq
		F
		=
		(Y\times F)_{ghg^{-1}}.
		\]
		The result follows from Theorem~\ref{thm:subrack-structure}.
	\end{proof}
	
	\begin{corollary}[A canonical product subposet]\label{cor:product-sublattice}
		Let
		\[
		\operatorname{Subrack}(\Conj(G))
		\]
		denote the poset of nonempty subracks of \(\Conj(G)\), and let
		\[
		\mathcal P(E)^G_{\neq\varnothing}
		:=
		\{F\subseteq E\mid F\neq\varnothing,\ g\bullet F=F\ \text{for all }g\in G\}
		\]
		be the poset of nonempty \(G\)-stable subsets of \(E\). Then the map
		\[
		\operatorname{Subrack}(\Conj(G))
		\times
		\mathcal P(E)^G_{\neq\varnothing}
		\longrightarrow
		\operatorname{Subrack}(\Conj(D)),
		\qquad
		(Y,F)\longmapsto Y\times F,
		\]
		is injective and order-preserving.
	\end{corollary}
	
	\begin{proof}
		If \(F\) is \(G\)-stable, then it is stable under the subgroup generated by
		any subrack \(Y\subseteq G\). Therefore
		Corollary~\ref{cor:product-subracks} shows that \(Y\times F\) is a subrack
		of \(\Conj(D)\).
		
		The map is order-preserving because
		\[
		Y_1\subseteq Y_2,\quad F_1\subseteq F_2
		\qquad\Longrightarrow\qquad
		Y_1\times F_1\subseteq Y_2\times F_2.
		\]
		To prove injectivity, suppose that
		\[
		Y_1\times F_1=Y_2\times F_2.
		\]
		Projecting to the first coordinate gives
		\[
		Y_1=Y_2.
		\]
		Since \(Y_1\) is nonempty, choose \(g\in Y_1\). Then
		\[
		\{g\}\times F_1
		=
		(Y_1\times F_1)\cap(\{g\}\times E)
		=
		(Y_2\times F_2)\cap(\{g\}\times E)
		=
		\{g\}\times F_2.
		\]
		Hence \(F_1=F_2\). Thus the map is injective.
	\end{proof}
	
	\begin{remark}[Trivial action]\label{rem:trivial-action}
		If the action of \(G\) on \(E\) is trivial, then
		\[
		(g,\alpha)\triangleright(h,\beta)
		=
		(ghg^{-1},\beta).
		\]
		Hence \(\Conj(D)\) is the direct product of the ordinary conjugation rack
		\(\Conj(G)\) with the trivial rack on \(E\). In this case the combinatorial
		invariants simplify. For instance,
		\[
		\operatorname{Inn}(\Conj(D))
		\cong
		G/Z(G),
		\]
		and
		\[
		\lvert\operatorname{Fix}(L_{(g,\alpha)})\rvert
		=
		|C_G(g)|\,|E|.
		\]
	\end{remark}
	
	\section{Group-like factorization and digroup algebras}
	
	We now return to cocommutative Hopf dialgebras and show that the rack
	constructed from the adjoint rack bialgebra is governed by the conjugation
	rack of the digroup of group-like elements.
	
	\subsection{The group-like functor and the factorization theorem}
	
	\begin{notation}\label{not:group-like-hopf-dialgebra}
		For a cocommutative Hopf dialgebra
		\[
		A=(A,\Delta,\varepsilon,\xi,\vdash,\dashv,S),
		\]
		we use the coalgebraic notation
		\[
		\Glike(A)
		=
		\{g\in A\mid \Delta(g)=g\otimes g,\ \varepsilon(g)=1_K\}
		\]
		for the set of group-like elements of its underlying coalgebra.
	\end{notation}
	
	Since \(\xi\) is group-like, one has \(\xi\in\mathrm{Glike}(A)\).
	Moreover, because the two dialgebra products are coalgebra morphisms,
	the set of group-like elements is stable under both products. The next
	result shows that this stability is strong enough to produce a digroup,
	not merely a generalized digroup.
	
	\begin{proposition}\label{prop:glike-digroup}
		Let
		\[
		A=(A,\Delta,\varepsilon,\xi,\vdash,\dashv,S)
		\]
		be a cocommutative Hopf dialgebra. Then \(\mathrm{Glike}(A)\) is a
		digroup with distinguished bar-unit \(\xi\). More precisely:
		\begin{enumerate}
			\item \(\mathrm{Glike}(A)\) is closed under \(\vdash\) and \(\dashv\);
			\item \(\xi\) is a bar-unit for the restricted products;
			\item for every \(g\in\mathrm{Glike}(A)\), the element \(S(g)\) belongs
			to \(\mathrm{Glike}(A)\) and satisfies
			\[
			g\vdash S(g)=\xi=S(g)\dashv g.
			\]
		\end{enumerate}
		In particular, the inverse map of the digroup \(\mathrm{Glike}(A)\) is the
		restriction of the antipode \(S\).
	\end{proposition}
	
	\begin{proof}
		Let \(g,h\in\mathrm{Glike}(A)\). Since each product
		\(\star\in\{\vdash,\dashv\}\) is a coalgebra morphism, we have
		\[
		\Delta(g\star h)
		=
		\sum_{(g),(h)}
		(g_{(1)}\star h_{(1)})\otimes(g_{(2)}\star h_{(2)}).
		\]
		Because \(g\) and \(h\) are group-like, this becomes
		\[
		\Delta(g\star h)
		=
		(g\star h)\otimes(g\star h).
		\]
		Similarly, since the counit is multiplicative with respect to both products,
		\[
		\varepsilon(g\star h)=\varepsilon(g)\varepsilon(h)=1.
		\]
		Thus \(g\vdash h\) and \(g\dashv h\) belong to \(\mathrm{Glike}(A)\).
		
		The dialgebra identities hold on all of \(A\), hence they restrict to
		\(\mathrm{Glike}(A)\). Moreover, the bar-unit identities in \(A\) give
		\[
		\xi\vdash g=g=g\dashv \xi
		\qquad
		(\forall g\in\mathrm{Glike}(A)).
		\]
		Therefore \(\xi\) is a bar-unit for the restricted dialgebra structure.
		
		It remains to check the inverse condition. Let \(g\in\mathrm{Glike}(A)\).
		Since \(S\) is a coalgebra morphism, we have
		\[
		\Delta(S(g))
		=
		(S\otimes S)\Delta(g)
		=
		(S\otimes S)(g\otimes g)
		=
		S(g)\otimes S(g),
		\]
		and
		\[
		\varepsilon(S(g))=\varepsilon(g)=1.
		\]
		Hence \(S(g)\in\mathrm{Glike}(A)\).
		
		Now apply the antipode identities
		\[
		\sum_{(a)}a_{(1)}\vdash S(a_{(2)})
		=
		\varepsilon(a)\xi,
		\qquad
		\sum_{(a)}S(a_{(1)})\dashv a_{(2)}
		=
		\varepsilon(a)\xi
		\]
		to \(a=g\). Since \(\Delta(g)=g\otimes g\), these identities reduce to
		\[
		g\vdash S(g)=\varepsilon(g)\xi=\xi
		\]
		and
		\[
		S(g)\dashv g=\varepsilon(g)\xi=\xi.
		\]
		Thus \(S(g)\) is simultaneously a right inverse of \(g\) for \(\vdash\)
		and a left inverse of \(g\) for \(\dashv\), with respect to the same
		bar-unit \(\xi\). Therefore \(\mathrm{Glike}(A)\) is a digroup.
	\end{proof}
	
	\begin{remark}\label{rem:inverse-notation-glike}
		When \(\mathrm{Glike}(A)\) is regarded as a digroup, we shall write
		\[
		g^{-1}:=S(g)
		\qquad
		(g\in\mathrm{Glike}(A)).
		\]
		Thus
		\[
		g\vdash g^{-1}=\xi=g^{-1}\dashv g.
		\]
	\end{remark}
	
	The previous proposition is stronger than what is needed to obtain a functor
	to generalized digroups. Indeed, the image lies in the full subcategory
	\(\mathbf{Dig}\subseteq\mathbf{gDig}\).
	
	\begin{proposition}\label{prop:glike-functor}
		The assignment
		\[
		A\longmapsto \mathrm{Glike}(A)
		\]
		extends to a functor
		\[
		\mathrm{Glike}:\mathbf{HopfDialg}\longrightarrow \mathbf{Dig}.
		\]
		Composing with the inclusion
		\[
		\mathbf{Dig}\hookrightarrow \mathbf{gDig},
		\]
		we also obtain a functor
		\[
		\mathrm{Glike}:\mathbf{HopfDialg}\longrightarrow \mathbf{gDig}.
		\]
	\end{proposition}
	
	\begin{proof}
		Let
		\[
		f:A\longrightarrow B
		\]
		be a morphism of cocommutative Hopf dialgebras. Thus \(f\) is a coalgebra
		morphism preserving the distinguished bar-unit, both products and the
		antipode.
		
		If \(g\in\mathrm{Glike}(A)\), then
		\[
		\Delta(f(g))
		=
		(f\otimes f)\Delta(g)
		=
		(f\otimes f)(g\otimes g)
		=
		f(g)\otimes f(g),
		\]
		and
		\[
		\varepsilon(f(g))=\varepsilon(g)=1.
		\]
		Hence \(f(g)\in\mathrm{Glike}(B)\). Moreover, for
		\(g,h\in\mathrm{Glike}(A)\),
		\[
		f(g\vdash h)=f(g)\vdash f(h),
		\qquad
		f(g\dashv h)=f(g)\dashv f(h),
		\]
		and
		\[
		f(S_A(g))=S_B(f(g)).
		\]
		Thus the restriction
		\begin{equation*}
			f\vert_{\Glike(A)}:
			\Glike(A)\longrightarrow \Glike(B)
		\end{equation*}
		is a digroup homomorphism. Identities and compositions are inherited from
		the category of cocommutative Hopf dialgebras. Therefore
		\(\Glike\) defines a functor into \(\mathbf{Dig}\), and hence also
		into \(\mathbf{gDig}\) by inclusion.
	\end{proof}
	
	We now compare the rack attached to a cocommutative Hopf dialgebra by the
	adjoint construction of Alexandre--Bordemann--Rivi\`ere--Wagemann with the
	conjugation rack of its digroup of group-like elements.
	
	\begin{theorem}[Factorization theorem]\label{thm:factorization-through-glike}
		Let
		\[
		\mathcal{R}:
		\mathbf{HopfDialg}
		\longrightarrow
		\mathbf{Rack}
		\]
		be the functor assigning to a cocommutative Hopf dialgebra \(A\) the rack of
		set-like elements of its adjoint rack bialgebra:
		\[
		\mathcal R(A):=\Slike(\Adj(A)).
		\]
		Then there is a natural isomorphism of functors
		\[
		\mathcal{R}
		\cong
		\mathrm{Conj}\circ\mathrm{Glike}.
		\]
		Equivalently, the rack functor associated with cocommutative Hopf dialgebras
		factorizes as
		\[
		\mathbf{HopfDialg}
		\xrightarrow{\ \mathrm{Glike}\ }
		\mathbf{Dig}
		\hookrightarrow
		\mathbf{gDig}
		\xrightarrow{\ \mathrm{Conj}\ }
		\mathbf{Rack}.
		\]
		More explicitly, for every cocommutative Hopf dialgebra \(A\) and every
		\(g,h\in\mathrm{Glike}(A)\), the rack product induced by the adjoint
		construction satisfies
		\[
		g\triangleright h
		=
		\sum_{(g)}(g_{(1)}\vdash h)\dashv S(g_{(2)})
		=
		(g\vdash h)\dashv S(g),
		\]
		and this is precisely the conjugation rack operation of the digroup
		\(\mathrm{Glike}(A)\).
	\end{theorem}
	
	\begin{proof}
		The rack bialgebra \(\mathrm{Adj}(A)\) has the same underlying coalgebra
		\[
		(A,\Delta,\varepsilon,\xi)
		\]
		as the Hopf dialgebra \(A\). Therefore its set-like elements are exactly the
		group-like elements of \(A\):
		\[
		\mathrm{Slike}(\mathrm{Adj}(A))
		=
		\mathrm{Glike}(A)
		\]
		as sets.
		
		We now compare the two rack operations on this common set. Let
		\(g,h\in\mathrm{Glike}(A)\). Since \(g\) is group-like,
		\[
		\Delta(g)=g\otimes g.
		\]
		Hence the adjoint rack product reduces to
		\[
		g\triangleright h
		=
		\sum_{(g)}(g_{(1)}\vdash h)\dashv S(g_{(2)})
		=
		(g\vdash h)\dashv S(g).
		\]
		By Proposition~\ref{prop:glike-digroup}, the inverse of \(g\) in the
		digroup \(\mathrm{Glike}(A)\) is \(g^{-1}=S(g)\). Therefore
		\[
		g\triangleright h
		=
		(g\vdash h)\dashv g^{-1},
		\]
		which is exactly the conjugation rack product on
		\(\mathrm{Conj}(\mathrm{Glike}(A))\).
		
		Thus the identity map on the underlying set defines a rack isomorphism
		\[
		\eta_A:
		\mathrm{Slike}(\mathrm{Adj}(A))
		\xrightarrow{\ \sim\ }
		\mathrm{Conj}(\mathrm{Glike}(A)).
		\]
		If \(f:A\to B\) is a morphism of cocommutative Hopf dialgebras, then the
		restriction of \(f\) to group-like elements is the same map used by both
		constructions. Consequently, the family \(\{\eta_A\}_A\) is natural in
		\(A\). This proves the stated natural factorization.
	\end{proof}
	
	\begin{corollary}\label{cor:factorization-through-digroups}
		The functor
		\[
		\mathcal{R}:
		\mathbf{HopfDialg}\longrightarrow\mathbf{Rack},
		\qquad
		\mathcal R(A)=\Slike(\Adj(A)),
		\]
		actually factorizes through digroups:
		\[
		\mathbf{HopfDialg}
		\xrightarrow{\ \mathrm{Glike}\ }
		\mathbf{Dig}
		\xrightarrow{\ \mathrm{Conj}\ }
		\mathbf{Rack}.
		\]
		In particular, the passage through generalized digroups is obtained by the
		standard inclusion
		\[
		\mathbf{Dig}\hookrightarrow\mathbf{gDig}.
		\]
	\end{corollary}
	
	\begin{corollary}\label{cor:combinatorics-for-group-likes}
		Let \(A\) be a cocommutative Hopf dialgebra such that
		\(\mathrm{Glike}(A)\) is finite. Let \(\xi\) be its distinguished bar-unit,
		and write
		\[
		\mathrm{Glike}(A)\simeq G\times E
		\]
		according to the structure theorem for generalized digroups, where
		\(G=G_l^\xi\) and \(E=\mathrm{Halo}(\mathrm{Glike}(A))\). Then the rack
		\(\mathcal{R}(A)\) is identified with the finite rack on \(G\times E\)
		given by
		\[
		(g,\alpha)\triangleright(h,\beta)
		=
		(ghg^{-1},g\bullet\beta).
		\]
		Consequently, the finite rack invariants of conjugation racks of generalized
		digroups, including the inner group, the left-translation cycle index, the
		fixed-point polynomial, orbit formulas and subrack structure, apply directly
		to \(\mathcal{R}(A)\).
	\end{corollary}
	
	\begin{proof}
		By Theorem~\ref{thm:factorization-through-glike}, we have a natural rack
		isomorphism
		\[
		\mathcal{R}(A)
		\cong
		\mathrm{Conj}(\mathrm{Glike}(A)).
		\]
		Applying the structure theorem for generalized digroups to
		\(\mathrm{Glike}(A)\), with respect to the bar-unit \(\xi\), gives
		\[
		\mathrm{Glike}(A)\simeq G\times E.
		\]
		Under this identification, the conjugation rack product is
		\[
		(g,\alpha)\triangleright(h,\beta)
		=
		(ghg^{-1},g\bullet\beta).
		\]
		The final statement follows by applying the finite combinatorial results
		for such racks.
	\end{proof}
	
	\begin{remark}[On the one-antipode hypothesis]\label{rem:one-antipode-hypothesis}
		The factorization above uses the balanced one-antipode form of a
		cocommutative Hopf dialgebra. In this setting the same coalgebra morphism
		\(S:A\to A\) satisfies
		\[
		\sum_{(a)}a_{(1)}\vdash S(a_{(2)})=\varepsilon(a)\xi,
		\qquad
		\sum_{(a)}S(a_{(1)})\dashv a_{(2)}=\varepsilon(a)\xi.
		\]
		This is precisely what ensures that each group-like element \(g\) has a
		single inverse \(S(g)\) in \(\mathrm{Glike}(A)\). In a two-antipode
		variant, one would expect the corresponding group-like object to be a
		genuine generalized digroup, with possibly distinct left and right inverse
		data.
	\end{remark}
	
	\subsection{The digroup algebra and the factorization functor}
	
	A classical source of cocommutative Hopf algebras is the group algebra construction (\cite[Part I (2.3)]{Jantzen2003})
	\[
	G\longmapsto K[G].
	\]
	We now describe the corresponding construction in the dialgebraic setting.
	Here the starting point is not a group, but a digroup. Thus \(D\) is endowed
	with two associative products \(\vdash\) and \(\dashv\), a distinguished
	bar-unit \(\xi\), and an inverse map \(x\mapsto x^{-1}\) satisfying
	\[
	x\vdash x^{-1}=\xi=x^{-1}\dashv x
	\qquad
	(x\in D).
	\]
	The purpose of this subsection is to show that the linearization of a digroup (\cite[Section 4]{RSV} and \cite{RSSV})
	is a natural source of cocommutative Hopf dialgebras, and that the
	factorization theorem recovers the original digroup without loss of
	information.
	
	No finiteness assumption on \(D\) is needed, since \(K[D]\) will denote the free
	vector space of finite linear combinations of elements of \(D\).
	
	\begin{proposition}\label{prop:digroup-algebra-hopf-dialgebra}
		Let \(D\) be a digroup with products \(\vdash,\dashv\), distinguished
		bar-unit \(\xi\), and inverse map \(x\mapsto x^{-1}\). Let
		\[
		K[D]:=\bigoplus_{x\in D}Kx
		\]
		be the free \(K\)-vector space with basis \(D\). Extend the two products
		of \(D\) bilinearly to maps
		\[
		\vdash,\dashv:K[D]\otimes K[D]\longrightarrow K[D].
		\]
		For \(x\in D\), define
		\[
		\Delta(x)=x\otimes x,
		\qquad
		\varepsilon(x)=1,
		\qquad
		S(x)=x^{-1},
		\]
		and extend \(\Delta\), \(\varepsilon\), and \(S\) linearly to \(K[D]\).
		Then
		\[
		\bigl(K[D],\Delta,\varepsilon,\xi,\vdash,\dashv,S\bigr)
		\]
		is a cocommutative Hopf dialgebra.
	\end{proposition}
	
	\begin{proof}
		We verify the defining properties on basis elements. The corresponding
		identities on all of \(K[D]\) then follow by multilinearity.
		
		First, since \(D\) is a digroup, both \((D,\vdash)\) and \((D,\dashv)\)
		are semigroups, and the mixed dialgebra identities
		\[
		x\vdash (y\dashv z)=(x\vdash y)\dashv z,
		\]
		\[
		x\dashv (y\dashv z)=x\dashv (y\vdash z),
		\]
		and
		\[
		(x\vdash y)\vdash z=(x\dashv y)\vdash z
		\]
		hold for all \(x,y,z\in D\). Extending the products bilinearly, these same
		identities hold on \(K[D]\). Hence \(K[D]\) is a dialgebra.
		
		The distinguished bar-unit \(\xi\in D\) satisfies
		\[
		\xi\vdash x=x=x\dashv \xi
		\qquad
		(x\in D).
		\]
		Moreover, the balanced identity
		\[
		x\vdash \xi=\xi\dashv x
		\qquad
		(x\in D)
		\]
		follows from the inverse identities. Indeed, since
		\[
		x^{-1}\dashv x=\xi,
		\]
		we have
		\[
		x\vdash \xi
		=
		x\vdash (x^{-1}\dashv x)
		=
		(x\vdash x^{-1})\dashv x
		=
		\xi\dashv x.
		\]
		Thus \(K[D]\) is a balanced bar-unital dialgebra.
		
		We now check the coalgebra structure. For \(x\in D\),
		\[
		(\Delta\otimes \id)\Delta(x)
		=
		(\Delta\otimes \id)(x\otimes x)
		=
		x\otimes x\otimes x,
		\]
		and similarly
		\[
		(\id\otimes \Delta)\Delta(x)
		=
		x\otimes x\otimes x.
		\]
		Hence \(\Delta\) is coassociative. Also,
		\[
		(\varepsilon\otimes \id)\Delta(x)=\varepsilon(x)x=x,
		\qquad
		(\id\otimes \varepsilon)\Delta(x)=x\varepsilon(x)=x.
		\]
		Therefore \((K[D],\Delta,\varepsilon)\) is a coalgebra. It is
		cocommutative because
		\[
		\tau\Delta(x)=\tau(x\otimes x)=x\otimes x=\Delta(x).
		\]
		Furthermore,
		\[
		\Delta(\xi)=\xi\otimes\xi,
		\qquad
		\varepsilon(\xi)=1,
		\]
		so \(\xi\) is group-like.
		
		Next we show that the two products are coalgebra morphisms. Let
		\(\star\in\{\vdash,\dashv\}\). For \(x,y\in D\), one has
		\[
		\Delta(x\star y)
		=
		(x\star y)\otimes(x\star y).
		\]
		On the other hand,
		\begin{align*}
			(\star\otimes \star)
			\circ(\id\otimes\tau\otimes\id)
			\bigl(\Delta(x)\otimes\Delta(y)\bigr)
			&=
			(\star\otimes \star)
			\circ(\id\otimes\tau\otimes\id)
			\bigl((x\otimes x)\otimes(y\otimes y)\bigr)\\
			&=
			(\star\otimes \star)(x\otimes y\otimes x\otimes y)\\
			&=
			(x\star y)\otimes(x\star y).
		\end{align*}
		Thus \(\Delta\circ\star\) satisfies the required compatibility. Similarly,
		\[
		\varepsilon(x\star y)=1=\varepsilon(x)\varepsilon(y).
		\]
		Hence both \(\vdash\) and \(\dashv\) are coalgebra morphisms.
		
		Now consider the antipode map. Since \(S(x)=x^{-1}\), we have
		\[
		\Delta(S(x))
		=
		\Delta(x^{-1})
		=
		x^{-1}\otimes x^{-1}
		=
		(S\otimes S)(x\otimes x)
		=
		(S\otimes S)\Delta(x),
		\]
		and
		\[
		\varepsilon(S(x))=\varepsilon(x^{-1})=1=\varepsilon(x).
		\]
		Thus \(S\) is a coalgebra morphism.
		
		Finally, the inverse identities in \(D\) give, for every \(x\in D\),
		\[
		\sum_{(x)}x_{(1)}\vdash S(x_{(2)})
		=
		x\vdash S(x)
		=
		x\vdash x^{-1}
		=
		\xi
		=
		\varepsilon(x)\xi,
		\]
		and
		\[
		\sum_{(x)}S(x_{(1)})\dashv x_{(2)}
		=
		S(x)\dashv x
		=
		x^{-1}\dashv x
		=
		\xi
		=
		\varepsilon(x)\xi.
		\]
		By linearity, both antipode identities hold for every element of \(K[D]\).
		Therefore
		\[
		\bigl(K[D],\Delta,\varepsilon,\xi,\vdash,\dashv,S\bigr)
		\]
		is a cocommutative Hopf dialgebra.
	\end{proof}
	
	\begin{definition}\label{def:digroup-algebra}
		The cocommutative Hopf dialgebra \(K[D]\) constructed above will be called
		the \textbf{digroup algebra} of \(D\).
	\end{definition}
	
	The next result shows that the group-like elements of \(K[D]\) recover the
	original digroup.
	
	\begin{proposition}\label{prop:group-likes-recover-digroup}
		Let \(D\) be a digroup and let \(K[D]\) be its digroup algebra. Then
		\[
		\Glike(K[D])=D,
		\]
		where \(D\) is identified with the canonical basis of \(K[D]\). Under this
		identification, the products induced on \(\Glike(K[D])\) coincide with the
		original digroup products on \(D\), and the inverse induced by the antipode
		coincides with the original inverse map of \(D\).
	\end{proposition}
	
	\begin{proof}
		Every basis element \(d\in D\) is group-like because
		\[
		\Delta(d)=d\otimes d,
		\qquad
		\varepsilon(d)=1.
		\]
		Hence
		\[
		D\subseteq \Glike(K[D]).
		\]
		
		Conversely, let
		\[
		x=\sum_{d\in D}\lambda_d\,d\in K[D]
		\]
		be group-like. Since elements of \(K[D]\) are finite linear combinations
		of basis elements, only finitely many coefficients \(\lambda_d\) are nonzero.
		The condition \(\Delta(x)=x\otimes x\) gives
		\[
		\sum_{d\in D}\lambda_d(d\otimes d)
		=
		\sum_{d,e\in D}\lambda_d\lambda_e(d\otimes e).
		\]
		Comparing coefficients in the basis
		\[
		\{d\otimes e\mid d,e\in D\}
		\]
		of \(K[D]\otimes K[D]\), we obtain
		\[
		\lambda_d\lambda_e=0
		\qquad
		(d\neq e).
		\]
		Thus at most one coefficient \(\lambda_d\) is nonzero. Therefore
		\(x=\lambda d\) for some \(d\in D\) and some \(\lambda\in K\). The
		group-like condition now becomes
		\[
		\lambda(d\otimes d)=\lambda^2(d\otimes d),
		\]
		so \(\lambda=\lambda^2\). Since \(K\) is a field, this implies
		\[
		\lambda\in\{0,1\}.
		\]
		But a group-like element also satisfies
		\[
		1=\varepsilon(x)=\varepsilon(\lambda d)=\lambda.
		\]
		Hence \(x=d\). Therefore
		\[
		\Glike(K[D])=D.
		\]
		
		The compatibility with \(\vdash\), \(\dashv\), and the inverse map follows
		immediately from the construction of \(K[D]\): the products are the bilinear
		extensions of the products of \(D\), and the antipode satisfies
		\[
		S(d)=d^{-1}
		\qquad
		(d\in D).
		\]
	\end{proof}
	
	\begin{theorem}\label{thm:factorization-on-digroup-algebra}
		Let \(D\) be a digroup and let \(K[D]\) be its digroup algebra. Then the
		rack associated with \(K[D]\) by the adjoint construction coincides, on
		group-like elements, with the conjugation rack of \(D\). More precisely,
		for all \(x,y\in D\),
		\[
		x\triangleright y
		=
		\sum_{(x)}(x_{(1)}\vdash y)\dashv S(x_{(2)})
		=
		(x\vdash y)\dashv x^{-1}.
		\]
		Consequently, the factorization functor recovers the original digroup and
		its conjugation rack:
		\[
		K[D]
		\longmapsto
		\Glike(K[D])=D
		\longmapsto
		\Conj(D).
		\]
	\end{theorem}
	
	\begin{proof}
		By Proposition~\ref{prop:group-likes-recover-digroup}, the group-like
		elements of \(K[D]\) are exactly the basis elements indexed by \(D\).
		Thus, for \(x\in D\),
		\[
		\Delta(x)=x\otimes x.
		\]
		Substituting this into the adjoint rack formula gives
		\[
		x\triangleright y
		=
		\sum_{(x)}(x_{(1)}\vdash y)\dashv S(x_{(2)})
		=
		(x\vdash y)\dashv S(x).
		\]
		By construction of the digroup algebra,
		\[
		S(x)=x^{-1}.
		\]
		Therefore
		\[
		x\triangleright y=(x\vdash y)\dashv x^{-1},
		\]
		which is precisely the conjugation rack operation of the digroup \(D\).
		The final statement follows again from
		\[
		\Glike(K[D])=D.
		\]
	\end{proof}
	
	\begin{remark}\label{rem:digroup-algebra-no-loss}
		Theorem~\ref{thm:factorization-on-digroup-algebra} shows that the digroup
		algebra construction does not lose the underlying rack-theoretic
		information. The intermediate object of group-like elements is precisely
		the original digroup, and the rack obtained from the adjoint construction is
		its conjugation rack. Thus the assignment
		\[
		D\longmapsto K[D]
		\]
		plays, in the dialgebraic context, the same structural role that the group
		algebra construction plays for ordinary groups.
	\end{remark}
	
	We now record a concrete family of examples coming from Kinyon's construction
	\cite{Kinyon}.
	
	\begin{example}\label{ex:kinyon-action-digroup}
		Let \(H\) be a group acting on the left on a set \(M\). Suppose that
		\(M\) contains a fixed point \(e\), that is,
		\[
		h\cdot e=e
		\qquad
		(h\in H).
		\]
		Then
		\[
		D:=M\times H
		\]
		becomes a digroup with products
		\[
		(u,h)\vdash(v,k)=(h\cdot v,hk),
		\qquad
		(u,h)\dashv(v,k)=(u,hk),
		\]
		distinguished bar-unit
		\[
		\xi=(e,1),
		\]
		and inverse
		\[
		(u,h)^{-1}=(e,h^{-1}).
		\]
		Therefore \(K[D]\) is a cocommutative Hopf dialgebra. On basis elements,
		the associated rack operation is
		\[
		(u,h)\triangleright(v,k)
		=
		\bigl((u,h)\vdash(v,k)\bigr)\dashv (u,h)^{-1}.
		\]
		Using the two products, we obtain
		\[
		(u,h)\triangleright(v,k)
		=
		(h\cdot v,hkh^{-1}).
		\]
	\end{example}
	
	\begin{example}\label{ex:C2-digroup}
		Let
		\[
		H=C_2=\{1,s\}
		\]
		act on
		\[
		M=\{e,a,b\}
		\]
		by fixing \(e\) and exchanging \(a\) and \(b\). Thus
		\[
		s\cdot e=e,
		\qquad
		s\cdot a=b,
		\qquad
		s\cdot b=a.
		\]
		Let
		\[
		D=M\times C_2
		\]
		be the corresponding digroup. Its digroup algebra \(K[D]\) is a
		cocommutative Hopf dialgebra with
		\[
		\Delta(u,h)=(u,h)\otimes(u,h),
		\qquad
		\varepsilon(u,h)=1,
		\qquad
		S(u,h)=(e,h^{-1}).
		\]
		The rack operation on \(D=\Glike(K[D])\) is
		\[
		(u,h)\triangleright(v,k)
		=
		(h\cdot v,hkh^{-1}).
		\]
		Since \(C_2\) is abelian, this simplifies to
		\[
		(u,h)\triangleright(v,k)
		=
		(h\cdot v,k).
		\]
		In particular,
		\[
		(a,s)\triangleright(a,1)=(b,1),
		\qquad
		(a,s)\triangleright(b,1)=(a,1),
		\]
		while
		\[
		(a,1)\triangleright(v,k)=(v,k)
		\qquad
		((v,k)\in D).
		\]
		Thus the conjugation rack is nontrivial even though the acting group
		\(C_2\) is abelian.
	\end{example}
	
	\begin{remark}\label{rem:two-antipode-outlook}
		The construction above uses digroups, hence it belongs to the balanced
		one-antipode setting. For a generalized digroup \(D\), after fixing a
		bar-unit \(\xi\), one may still form the vector space
		\[
		K[D]=\bigoplus_{x\in D}Kx
		\]
		and define
		\[
		\Delta(x)=x\otimes x,
		\qquad
		\varepsilon(x)=1.
		\]
		The two products of \(D\) again extend bilinearly to \(K[D]\), producing a
		cocommutative bar-unital di-coalgebra. The difference appears at the
		antipode level. In a generalized digroup, the inverse data relative to
		\(\xi\) are generally one-sided:
		\[
		x\vdash x^{-1}_{r,\xi}=\xi,
		\qquad
		x^{-1}_{l,\xi}\dashv x=\xi,
		\]
		and the two elements \(x^{-1}_{r,\xi}\) and \(x^{-1}_{l,\xi}\) need not
		coincide. Thus the natural linear maps are
		\[
		S_{\vdash}(x):=x^{-1}_{r,\xi},
		\qquad
		S_{\dashv}(x):=x^{-1}_{l,\xi},
		\]
		rather than a single antipode. This suggests a two-antipode version of the
		linearization construction,
		\[
		\text{generalized digroups}
		\longrightarrow
		\text{two-antipode Hopf dialgebraic objects}
		\longrightarrow
		\mathbf{Rack}.
		\]
		In the digroup case the two inverse maps coincide, and the construction
		collapses to the one-antipode Hopf dialgebra \(K[D]\) studied above.
	\end{remark}
	
	\section{Conclusion and outlook}
	
	We have shown that the rack associated with a cocommutative Hopf dialgebra is
	controlled by an intermediate digroup of group-like elements. More precisely,
	the adjoint rack bialgebra construction, after restriction to set-like
	elements, is naturally isomorphic to the conjugation rack of
	\(\Glike(A)\). This gives a functorial factorization
	\[
	\mathbf{HopfDialg}
	\xrightarrow{\ \Glike\ }
	\mathbf{Dig}
	\hookrightarrow
	\mathbf{gDig}
	\xrightarrow{\ \Conj\ }
	\mathbf{Rack}.
	\]
	
	For finite generalized digroups \(D\simeq G\times E\), this factorization leads
	to explicit rack-combinatorial formulas. The left translations are products
	\[
	L_{(g,\alpha)}=c_g\times \rho(g),
	\]
	and this allows one to compute the inner permutation group, the
	left-translation cycle index, fixed-point data, orbit numbers and subrack
	structure in terms of group conjugation and the action of \(G\) on the halo
	\(E\).
	
	Finally, the digroup algebra construction \(D\mapsto K[D]\) provides a
	dialgebraic analogue of the classical group algebra construction. Its
	group-like elements recover the original digroup, and the rack obtained from
	the adjoint construction is exactly the conjugation rack of \(D\). This shows
	that the passage from digroups to cocommutative Hopf dialgebras preserves the
	rack-theoretic information carried by conjugation.
	
	A natural direction for future work is the development of a two-antipode
	version of this theory. For generalized digroups, right and left inverse data
	need not coincide, and therefore the linearized object should involve two
	antipode-like maps. Such a theory would extend the present balanced
	one-antipode framework and may provide a Hopf-dialgebraic setting in which
	generalized digroups, rather than only digroups, appear directly as group-like
	objects.

\end{document}